\documentclass[10pt]{article}
\usepackage{bbm}
\usepackage{mathrsfs}
\usepackage{amsfonts}
\usepackage{amsthm,amsmath,amssymb,anysize}
\newtheorem{lemma}{Lemma}[section]
\newtheorem{theorem}[lemma]{Theorem}
\newtheorem{remark}[lemma]{Remark}
\newtheorem{proposition}[lemma]{Proposition}

\usepackage[numbers,sort&compress]{natbib}
\marginsize{3.2cm}{3.2cm}{3.6cm}{3cm}
\numberwithin{equation}{section}

\title{\textsf{ The first cohomology of $\frak{sl}(2,1)$
with coefficients in $\chi$-reduced Kac modules and simple modules\footnote{Supported by  the NSF
  of China (11701158, 11601135).}}}
\author{\textsc{Shujuan Wang$^{1}$ and
  \textsc{Wende Liu$^{2,}$}}\footnote{Correspondence:  wendeliu@ustc.edu.cn (W. Liu)}\\
  {\small \textit{$^1$School of Mathematical Sciences, Heilongjiang University,}}\\
\small \textit{Harbin 150080, China} \\
 \small\textit{$^2$School of Mathematics and Statistics,}
\textit{Hainan Normal University,}\\ \small\textit{ Haikou 571158,  China} }
\date{ }
\begin{document}
\maketitle
\begin{quotation}
\small\noindent \textbf{Abstract}:
Over a field of characteristic $p>2,$
 the first cohomology of  the special linear Lie superalgebra $\frak{sl}(2,1)$ with
coefficients in all  $\chi$-reduced Kac modules and simple modules is determined by use of  the weight space decompositions of these modules
relative to the standard Cartan subalgebra of $\frak{sl}(2,1).$

 \vspace{0.2cm} \noindent{\textbf{Keywords}}: $\frak{sl}(2, 1)$, $\chi$-reduce Kac modules, simple $\chi$-modules, cohomology.

\vspace{0.1cm} \noindent \textbf{Mathematics Subject Classification
2000}: 17B50, 17B40

\end{quotation}
\setcounter{section}{0}
\section{Introduction}

Let $\frak{g}$ be a finite-dimensional simple Lie algebra and $M$ a finite-dimensional simple $\frak{g}$-module. 
Over the complex number field $\mathbb{C}$, it is well-known that cohomology $\mathrm{H}^*(\frak{g}, M) = 0$ (see \cite{GL}). 
However, there are very few general theorems helping to compute $\mathrm{H}^*(\frak{g}, M)$ when 
both $\frak{g}$ and $M$ are considered over a field of prime characteristic, even less in the Lie superalgebra situation.
For nontrivial modules of Lie superalgebras, we have only the
same tool, that is the definition, used by researchers  at the birth of the cohomology theory.
Very little seems to be known about the study of modular Lie superalgebra cohomology when coefficient modules are not trivial or adjoint ones,
except \cite{zls, liusun, sunliu}.
In \cite{Kac1,Kac2}, Kac posed the problem of determining the first cohomology of the simple Lie superalgebras with coefficients in the simple modules over a field of characteristic 0.
 By use of the fact that finite-dimensional modules of reductive Lie algebras are  completely reducible, Su and Zhang computed explicitly first cohomology  of the special linear  Lie superalgebra $\frak{sl}(m,n)$ over  a field of characteristic 0
with coefficients in Kac modules and simple modules (see \cite{suzhang}).
In the case of prime characteristic, however, the  complete reducibility  is not true.
In this paper, we overcome this difficulty by
considering  weight-derivations to  determine  the first cohomology of  $\frak{sl}(2,1)$  with coefficients in all  $\chi$-reduced Kac modules and simple modules over a field of characteristic $p>2$.
Lie superalgebra cohomology  is of crucial importance to understand  extensions of modules,
and also extensions of the Lie superalgebras themselves. Moreover,
 it is  an independent topic  in the whole Lie superalgebra theory.
For example, relative  cohomology  plays a central role in the Borel-Weil-Bott theory (see \cite{Penkov});
cohomology  of nilpotent radicals of parabolic subalgebras is key in the Kazhdan-Lusztig theory (see \cite{Brundan}).

  Considering that the simple modules of  $\frak{sl}(2,1)$ are completely determined  over a field of prime characteristic (see \cite[Theorem 3.13]{zhang}),
  we aim to deal with the modular-versions
of Kac's problem on  $\frak{sl}(2,1)$. Our main results are as follows, which are quite different from the corresponding ones of $\frak{sl}(2,1)$ over a field of characteristic 0 (see \cite[Theorems 1.1 and 1.2]{suzhang}).

Hereafter we assume that the underlying field $\mathbb{F}$ is algebraically closed and of characteristic  $p>2$.
Our main results are as follows:
\begin{theorem}\label{1903221616}
 Let $Z^{\chi}(\lambda)$ be the $\chi$-reduced Kac module of $\frak{sl}(2,1)$ with the highest weight $\lambda$. Then
$$\dim\mathrm{H}^1(Z^{\chi}(\lambda))=\left\{\begin{array}{ll}
1,\;\;&\mbox{if}\;\;\chi=0,\;\;\mbox{and}\;\;\lambda=(p-1, p-2)\;\;\mbox{or}\;\;(p-2, 0)\\
0,\;\;&\mbox{otherwise}.
\end{array}\right.$$
\end{theorem}
\begin{theorem}\label{1903231610}
Let $S^{\chi}(\lambda)$ be the  simple $\chi$-module of $\frak{sl}(2,1)$ with the highest weight $\lambda$. Then
$$\dim\mathrm{H}^1(S^{\chi}(\lambda))=\left\{\begin{array}{lll}
1,\;\;&\mbox{if}\;\;\chi=0,\;\;\mbox{and}\;\;\lambda=(p-1, p-1)\;\;\mbox{or}\;\;(1, 0)\\
2, \;\;&\mbox{if}\;\;\chi=0,\;\;\mbox{and}\;\;\lambda=(p-1, 0), \;\;\mbox{or}\;\;(p-1, 1)\\
0,\;\;&\mbox{otherwise}.
\end{array}\right.$$
\end{theorem}

\section{Preliminaries}
Suppose that $\frak{g}$ is a finite-dimensional  Lie superalgebra and $M$  a finite-dimensional $\frak{g}$-module.
Recall that a $\mathbb{Z}_{2}$-homogeneous linear mapping $\varphi: \frak{g}\longrightarrow M$
is a derivation of parity $|\varphi|$,  if
\begin{eqnarray*}\label{low00}
\varphi([x,y])=(-1)^{|\varphi| |x|}x \varphi(y)-(-1)^{|y|
(|\varphi|+|x|)}y \varphi(x)\; \mbox{for all}\; x, y\in \frak{g}.
\end{eqnarray*}
Hereafter  $\mathbb{Z}_2=\{\bar{0}, \bar{1}\}$ is the field of two elements and
  $|a|$ is the parity ($\mathbb{Z}_2$-degree) of $a$ in the superspace under consideration. In this paper, the symbol $|a|$ always implies that $a$ is $\mathbb{Z}_2$-homogeneous in a superspace.
Let $\mathrm{Der}(\frak{g},M)$  denote the superspace spanned by all the homogeneous derivations from $\frak{g}$ to $M$,
 each element in which is called a derivation.
 For $m\in M$, define the map $\frak{D}_m$ from $\frak{g}$ to $M$ by letting
 $$\frak{D}_m(x)=(-1)^{|x||m|}x m\; \mbox{for all}\; x\in \frak{g}.$$
 Then $\frak{D}_m$ is a homogeneous derivation of parity $|m|$.
 Write $\mathrm{Ider}(\frak{g},M)$ for the subspace spanned by all $\frak{D}_m$ with homogeneous elements  $m\in M$,
 each element in which is called an inner derivation.
 In general, the $\frak{g}$-module $\mathrm{Hom}_{\mathbb{F}}(\frak{g},M)$ contains
 $\mathrm{Ider}(\frak{g},M)$ and $\mathrm{Der}(\frak{g},M)$ as submodules.
 By definition, the
first cohomology of $\frak{g}$ with coefficients in $M$ is
\begin{equation}\label{1906131540}
\mathrm{H}^1(\frak{g},M)=\mathrm{Der}(\frak{g},M)/\mathrm{Ider}(\frak{g},M).
\end{equation}
Let $\frak{h}$ be a Cartan subalgebra  in the even part of $\frak{g}.$ Suppose
$\frak{g}$ and  $M$ possess  weight space decompositions with respect to $\frak{h}$:
$$\mbox{$\frak{g}=\oplus _{\gamma \in \frak{h}^*}\frak{g}_{\gamma}$ and $M=\oplus _{\gamma \in
\frak{h}^*}M_{\gamma}$.}$$
In this manner, $\frak{g}$ is also an $\frak{h}^*$-graded Lie superalgebra and $M$ an $\frak{h}^*$-graded module.
Then the assumption on the dimensions implies that the $\frak{g}$-module $\mathrm{Hom}_{\mathbb{F}}(\frak{g},M)$
has an $\frak{h}^*$-grading structure induced by the ones of $\frak{g}$ and $M$:
$$
\mathrm{Hom}_{\mathbb{F}}(\frak{g},M)=\oplus _{\gamma \in \frak{h}^*}\mathrm{Hom}_{\mathbb{F}}(\frak{g},M)_{(\gamma)},
$$
where
$$
\mathrm{Hom}_{\mathbb{F}}(\frak{g},M)_{(\gamma)}=\{\phi\in \mathrm{Hom}_{\mathbb{F}}(\frak{g},M)\mid \phi(\frak{g}_{\alpha})\subset M_{\alpha+\gamma}, \forall \alpha \in  \frak{h}^*\}.
$$
Moreover, $\mathrm{Hom}_{\mathbb{F}}(\frak{g},M)_{(\gamma)}$ is precisely  the weight space of weight $\gamma$ with respect to $\frak{h}$:
\begin{align*}
\mathrm{Hom}_{\mathbb{F}}(\frak{g},M)_{(\gamma)}&=\mathrm{Hom}_{\mathbb{F}}(\frak{g},M)_{\gamma}\\
&=\{\varphi\in \mathrm{Hom}_{\mathbb{F}}(\frak{g},M)\mid h \varphi=\gamma(h)\varphi, \forall h\in \frak{h}\}.
\end{align*}
It is a standard fact that $\mathrm{Der}(\frak{g},M)$ is an $\frak{h}^*$-graded submodule (hence a weight-submodule with respect to $\frak{h}$) of $\mathrm{Hom}_{\mathbb{F}}(\frak{g},M)$:
\begin{align}\label{derweightdec}
\mathrm{Der}(\frak{g},M)=\oplus _{\gamma \in \frak{h}^*}\mathrm{Der}(\frak{g},M)_{\gamma},
\end{align}
where
\begin{align*}
\mathrm{Der}(\frak{g},M)_{\gamma}=\mathrm{Der}(\frak{g},M)\cap \mathrm{Hom}_{\mathbb{F}}(\frak{g},M)_{\gamma}.
\end{align*}

A linear map  (resp. derivation) in $\mathrm{Hom}_{\mathbb{F}}(\frak{g},M)_{0}$ (resp. $\mathrm{Der}(\frak{g},M)_{0}$)  is called a \textit{weight-map} (resp. \textit{weight-derivation}) with respect to $\frak{h}$.

The following lemma provides a useful  reduction method in computing the first
cohomology, which is a super-version
of Lie algebra case \cite[Theorem 1.1]{RF} (see also \cite[Lemma 3.2]{Bai}): a derivation must be a weight-derivation modulo an inner derivation.
\begin{lemma}\label{lem-weight-der+inner}
Retain the above notations. Then
$$\mathrm{Der}(\frak{g},M)=\mathrm{Der}(\frak{g},M)_{0}+\mathrm{Ider}(\frak{g},M).$$
\end{lemma}
\begin{proof} In view of (\ref{derweightdec}), it suffices to show that a derivation of nonzero weight must be inner.
Suppose $\varphi\in \mathrm{Der}(\frak{g},M)_{\alpha}$ with  nonzero weight (degree)  $\alpha$.
Then there exists $h\in \frak{h}$ such that $\alpha(h)=1$.
It follows that $h\varphi=\alpha(h)\varphi=\varphi$.
Then, for all $x\in \frak{g}$, we have
$$\varphi(x)=(h\varphi)(x)=h\left(\varphi(x)\right)-\varphi\left([h, x]\right)=(-1)^{|x||\varphi|}x\left(\varphi(h)\right).$$
Therefore, $\varphi=\frak{D}_{\varphi(h)}$ is inner.
\end{proof}

Suppose further that $\frak{g}$ is  restricted  with  $p$-mapping $[p]$.
Any simple $\frak{g}$-module $M$ possesses a unique $p$-character $\chi\in \frak{g}^*_{\bar{0}}$,
that is, $x^p-x^{[p]}=\chi(x)^p\mathrm{id}_M$ on $M$ for all $x\in \frak{g}_{\bar{0}}$.
 We call $M$ a $\chi$-module of $\frak{g}$.
For $\chi\in \frak{g}_{\bar{0}}^*$, by definition the $\chi$-reduced enveloping algebra of $\frak{g}$ is $u(\frak{g}, \chi) = U(\frak{g})/I_{\chi}$,
where $I_{\chi}$ is the $\mathbb{Z}_2$-graded two-sided ideal of  $U(\frak{g})$ generated by elements $\left\{x^p-x^{[p]}-\chi(x)\mid x\in \frak{g}_{\bar{0}}\right\}$  (see \cite[p.757]{zhang}).   We often regard $\chi\in \frak{g}^*$ by letting $\chi(\frak{g}_{\bar{1}})=0$.
The category of simple $u(\frak{g}, \chi)$-modules may be identified with the one of simple $\chi$-modules
of $\frak{g}$. The adjoint algebraic group $G_{\bar{0}}$ of $\frak{g}_{\bar{0}}$ acts on $\frak{g}$ by adjoint action,
since $\frak{g}_{\bar{1}}$ is a rational $G_{\bar{0}}$-module.
If $\chi$ and $\chi'$ in the same orbit, then $u(\frak{g}, \chi)\cong u(\frak{g}, \chi')$ (see \cite[Remark 2.5]{wangwq}).


From now on,  we write $\frak{g}$ for the special linear Lie superalgebra $\frak{sl}(2,1)$, which is
the subalgebra of $\frak{gl}(2,1)$ consisting of supertraceless matrices.
The following elements form a basis of $\frak{g}$:
$$\mbox{$h_1:=e_{11}+e_{33}, \;h_2:=e_{22}+e_{33}$ and  $e_{ij}$ with $ 1\leq i\neq j\leq 3,$}$$
where $e_{kl}$ is the $3\times 3$ matrix unit.
Fix the standard Cartan subalgebra $\frak{h}$  of $\frak{g}_{\bar{0}}$ spanned by $h_1$ and $ h_2.$
If $\lambda\in \frak{h}^*$, then we write $(\lambda_1, \lambda_2)$ for $\lambda$, where $\lambda_1=\lambda(h_1), \lambda_2=\lambda(h_2)$.
For $\chi\in \frak{g}^*_{\bar{0}}$,
it is a standard fact that each simple $u(\frak{g}_{\bar{0}}, \chi)$-module $M$  is generated by the unique maximal
vector  of weight $\lambda$.  Then $\lambda_i^p-\lambda_i=\chi(h_i)^p$ for $i=1,2$, since $h_i^{[p]}=h_i$.
In this case we write $M(\lambda)$ instead of $M$ and call $\lambda$  the highest weight of $M(\lambda)$.
Note that $\frak{g}$ possesses  a $\mathbb{Z}$-grading structure $\frak{g}=\frak{g}_{-1}\oplus\frak{g}_{0}\oplus\frak{g}_{1},$ where $\frak{g}_{-1}=\mathbb{F}e_{31}+\mathbb{F}e_{32}, \frak{g}_0=\frak{g}_{\bar{0}}$ and $\frak{g}_1=\mathbb{F}e_{13}+\mathbb{F}e_{23}.$
Regarding $M(\lambda)$  as $u(\frak{g}_{0}\oplus\frak{g}_{1}, \chi)$-module with $\frak{g}_1$
acting trivially,  we may construct the induced module
$$Z^{\chi}(\lambda)=u(\frak{g}, \chi)\otimes_{u(\frak{g}_{0}\oplus\frak{g}_{1}, \chi)}M(\lambda),$$
which is called   $\chi$-reduced Kac module of $\frak{g}$ with the highest weight $\lambda$
(denoted by $Z^{\chi}(M)$ in \cite[p. 759]{zhang} for $M:=M(\lambda)$).
Since $\frak{g}_{\bar{0}}\cong \frak{gl}(2)$, there are two kinds of coadjoint orbits of $\frak{g}_{\bar{0}}^*$ with
representatives (see \cite[Sec. 5.4]{Jan}):
\begin{itemize}
\item[(1)] semi-simple: $\chi\left( h_1\right)= r, \chi\left( h_2\right)= s, \chi\left(e_{12}\right)= 0, \chi\left(e_{21}\right)= 0$.
\item[(2)] nilpotent: $\chi\left( h_1\right)= r, \chi\left( h_2\right)= r, \chi\left(e_{12}\right)= 0, \chi\left(e_{21}\right)= 1.$
\end{itemize}
If $\chi$ is nilpotent or semi-simple with $r=s$,
then
$$\lambda_1^p-\lambda_1=\chi(h_1)^p=\chi(h_2)^p=\lambda_2^p-\lambda_2$$
and hence $\left(\lambda_1-\lambda_2\right)^p=\lambda_1-\lambda_2.$
 It follows that $\lambda_1-\lambda_2$ is in $\mathbb{F}_p$.
 Hereafter $\mathbb{F}_p$ denotes the prime subfield of $\mathbb{F}$.
 In particular, if $\chi(h_1)=\chi(h_2)=0$, then $\lambda_1\in \mathbb{F}_p, \lambda_2\in \mathbb{F}_p$.
For convenience, write
$\langle i, j, k\rangle$ for the element $e_{31}^ie_{32}^j e_{21}^kv_0$ in $Z^{\chi}(\lambda)$,
where $v_0$ is the unique maximal vector of weight $\lambda$ in $M(\lambda)$.
Moreover, the symbol $\langle i, j, k\rangle$ always implies that $k$ is the smallest nonnegative integer
in the residue class containing $k$ modulo $p$.
When $\chi$ is semi-simple,    $Z^{\chi}(\lambda)$ has a basis
$$\left\{\langle i, j, k\rangle\mid \mbox{$i, j=0$ or $ 1$}, k=\left\{\begin{array}{ll}
0, 1, \ldots, \lambda_1-\lambda_2, &\mbox{if  $r=s$ and $\lambda_1-\lambda_2\neq p-1$}\\
0, 1, \ldots, p-1, &\mbox{if $r\neq s$ or $r=s$ but $\lambda_1-\lambda_2=p-1$}
\end {array}\right.\right\}.$$
When $\chi$ is nilpotent,   $Z^{\chi}(\lambda)$ has a basis
$$\left\{\langle i, j, k\rangle\mid \mbox{$i, j=0$ or $ 1$}, k=0, 1, \ldots, p-1\right\}.$$
Below  $\delta_{P}$ means 1 if a proposition $P$ is true, and 0 otherwise.
We list some formulas to be used in the future, which only need a direct computation:
\begin{align}
&h_1\langle i,j,k\rangle=(\lambda_1+j-k)\langle i,j,k\rangle\label{h1},\\
&h_2\langle i,j,k\rangle=(\lambda_2+i+k)\langle i,j,k\rangle\label{h2},\\
&e_{31}\langle i,j,k\rangle=\delta_{i=0}\langle 1,j,k\rangle\label{e31},\\
&e_{32}\langle i,j,k\rangle=(-1)^i\delta_{j=0}\langle i,1,k\rangle\label{e32},\\
&e_{12}\langle i,j,k\rangle=k(\lambda_1-\lambda_2-k+1)\langle i,j,k-1\rangle-\delta_{(i,j)=(1,0)}\langle0,1,k\rangle\label{e12},\\
&e_{13}\langle i,j,k\rangle=\delta_{j=1}(-1)^ik(\lambda_1-\lambda_2-k+1)\langle i,0,k-1\rangle+\delta_{i=1}(j+\lambda_1-k)\langle0,j,k\rangle\label{e13}.
\end{align}
In addition,  if  $\chi$ is semi-simple, then
\begin{align}
&e_{21}\langle i,j,k\rangle=-\delta_{(i,j)=(0,1)}\langle 1,0,k\rangle+\delta_{k\neq p-1}\delta_{k\neq \lambda_1-\lambda_2}\langle i,j,k+1\rangle\label{e21s},\\
&e_{23}\langle i,j,k\rangle=
-\delta_{j=1}\left((-1)^i(\lambda_2+k)-i\right)\langle i,0,k\rangle+\delta_{i=1}\delta_{k\neq p-1}\delta_{k\neq \lambda_1-\lambda_2}\langle 0,j,k+1\rangle\label{e23s}.
\end{align}
If  $\chi$ is nilpotent, then
\begin{align}
&e_{21}\langle i,j,k\rangle=-\delta_{(i,j)=(0,1)}\langle 1,0,k\rangle+\langle i,j,k+1\rangle\label{e21n},\\
&e_{23}\langle i,j,k\rangle=\delta_{j=1}\left((-1)^i(\lambda_2+k)-i\right)\langle i,0,k\rangle+\delta_{i=1}\langle 0,j,k+1\rangle\label{e23n}.
\end{align}

Recall that $Z^{\chi}(\lambda)$ possesses a unique  simple quotient module, denoted by  $S^{\chi}(\lambda)$.
It is a standard fact that each simple $\frak{g}$-module   must be of   form
$S^{\chi}(\lambda)$  with $\chi\in \frak{g}^*_{\bar{0}}$ and $\lambda\in \frak{h}^*$.
Without confusion, in this paper we also use $\langle i, j, k\rangle$
to represent the residue class of $e_{31}^ie_{32}^j e_{21}^kv_0$ in $S^{\chi}(\lambda)$.

Let us sketch some main points on $S^{\chi}(\lambda)$,
and leave out some details (see \cite[Theorem 3.13]{zhang}).
\begin{lemma}\label{2002081211}
\begin{itemize}
\item[(1)]
If $\lambda_1\neq p-1$ and $\lambda_2\neq 0$, then $S^{\chi}(\lambda)=Z^{\chi}(\lambda)$ for all $\chi\in \frak{g}_{\bar{0}}^*$.
\item[(2)]
If $\lambda_1= p-1$ and $\lambda_2\neq 0, p-1$, then $S^{0}(\lambda)$ has a basis $\langle0,0,k\rangle, \langle1,0,k\rangle, \langle0,1,0\rangle$,
and
$$\langle1,1,k\rangle=0, \;\langle0,1,k+1\rangle=(\lambda_2+k+1)\langle1,0,k\rangle,$$
where $0\leq k\leq p-1-\lambda_2$.
\item[(3)]
If $\lambda=(p-1, 0)$, then $S^{0}(\lambda)$ has a basis $\langle0,0,k\rangle, \langle1,0,k\rangle$,
 and
$$\langle1,1,k\rangle=0, \;\langle0,1,k\rangle=k\langle1,0,k-1\rangle,$$
where $0\leq k\leq p-1$.
\item[(4)]
If $\lambda=(p-1, p-1)$, then $S^{0}(\lambda)$ has a basis $\langle1,0,0\rangle, \langle0,0,0\rangle, \langle0,1,0\rangle$,
and
$$\mbox{$\langle1,1,0\rangle=\langle i,j,k\rangle=0$ for $i,j=0$ or 1, $1\leq k\leq p-1.$}$$
\item[(5)]
If $\lambda_2=0, \lambda_1\neq p-1$, then $S^{0}(\lambda)$ has a basis
$\langle0,0,k\rangle, \langle1,0,k\rangle, \langle0,0,\lambda_1\rangle$,
and
$$\langle0,1,k\rangle=k\langle1,0,k-1\rangle, \; \langle0,1,\lambda_1\rangle=\lambda_1\langle1,0,\lambda_1-1\rangle, \;\langle1,1,k\rangle=\langle1,1,\lambda_1\rangle=\langle1,0,\lambda_1\rangle=0,$$
where $0\leq k\leq \lambda_1-1$.
\item[(6)]
Suppose $\chi$ is nilpotent and $\chi(h_1)=\chi(h_2)=0$.
If $\lambda_1=p-1$ or $\lambda_2=0$, then $S^{0}(\lambda)$ has a basis
$\langle0,0,k\rangle, \langle1,0,k\rangle$,
and
$$\langle0,1,k\rangle=(\lambda_2+k)\langle1,0,k-1\rangle, \; \langle1,1,k\rangle=0,$$
where $0\leq k\leq p-1$.
\end{itemize}
\end{lemma}

Note that nonzero $\langle i, j, k\rangle$ is of   weight $(\lambda_1+j-k, \lambda_2+i+k)$ by (\ref{h1}) and (\ref{h2}),
 and the weight vectors in $\frak{g}$ are of  weights $(\mu_1, \mu_2)$, where $\mu_1, \mu_2, \mu_1+\mu_2\in \left\{0, 1, -1\right\}$.
 Therefore, if the weight of nonzero $\langle i, j, k\rangle$ is also one of $\frak{g}$,  then $\langle i, j, k\rangle$ must be one of the following:
\begin{equation}\label{1906211115}
\begin{split}
&\mbox{$\langle 0, 0, \lambda_1-i\rangle$  \;\ with weight \;\ $(i, \lambda_1+\lambda_2-i)$},\\
&\mbox{$\langle 0, 0, i-\lambda_2\rangle$  \;\ with weight \;\ $(\lambda_1+\lambda_2-i, i)$},\\
&\mbox{$\langle 1, 1, \lambda_1+j\rangle$  \;\  with weight \;\ $(1-j, \lambda_1+\lambda_2+j+1)$},\\
&\mbox{$\langle 1, 1, -\lambda_2-j\rangle$ \;\  with weight \;\ $(\lambda_1+\lambda_2+j+1, 1-j)$},\\
&\mbox{$\langle 1, 0, \lambda_1-i\rangle$ and $\langle 0, 1, \lambda_1-i+1\rangle$   \;\ with weight \;\  $(i, \lambda_1+\lambda_2+1-i)$},\\
&\mbox{$\langle 1, 0, -\lambda_2-j\rangle$ and $\langle 0, 1, -\lambda_2-j+1\rangle$ \;\ with weight \;\  $(\lambda_1+\lambda_2+j, 1-j)$},\\
\end{split}
\end{equation}
where $i=0,-1,1, j=0, 1, 2$.
Note that $S^{\chi}(\lambda)$ is the unique  simple quotient module of $Z^{\chi}(\lambda)$. Then we have the following remark, which will be frequently used.

\begin{remark}\label{1904202222}
If $\varphi$ is a weight-map from $\frak{g}$ to $Z^{\chi}(\lambda)$ or $S^{\chi}(\lambda)$,
then  $\varphi$ is of the following forms and we  retain throughout the notations therein:

(1) When $\lambda_1+\lambda_2=1$,
$$\varphi\left(e_{31}\right)= a\langle 0, 0, \lambda_1\rangle,\;
\varphi\left(e_{32}\right)= b\langle 0, 0, \lambda_1-1\rangle, \;
\varphi\left(x\right)= 0,$$
where $x=h_1, h_2, e_{ij}$ with $(i, j)\neq(3, 1), (3, 2)$, and $a, b\in \mathbb{F}$.

(2) When $\lambda_1+\lambda_2=-1$,
\begin{align*}
&\varphi\left(h_1\right)= a_1\langle 1, 0, \lambda_1\rangle+a_2\langle 0, 1, \lambda_1+1\rangle, \\
&\varphi\left(h_2\right)= a_3\langle 1, 0, \lambda_1\rangle+a_4\langle 0, 1, \lambda_1+1\rangle,\\
&\varphi\left(e_{12}\right)= a_5\langle1, 0, \lambda_1-1\rangle+a_6\langle 0, 1, \lambda_1\rangle,\\
&\varphi\left(e_{21}\right)= a_9\langle 1, 0, \lambda_1+1\rangle+a_{10}\langle 0, 1, \lambda_1+2\rangle,\\
&\varphi\left(e_{13}\right)= a_7\langle 0, 0, \lambda_1\rangle, \;\varphi\left(e_{23}\right)= a_8\langle 0, 0, \lambda_1+1\rangle,\\
&\varphi\left(e_{31}\right)= a_{11}\langle 1, 1, \lambda_1+1\rangle, \;\varphi\left(e_{32}\right)= a_{12}\langle 1, 1, \lambda_1\rangle,
\end{align*}
where $a_i\in \mathbb{F}$ for $1\leq i\leq 12$.

(3) When $\lambda_1+\lambda_2=0$,
\begin{align*}
&\varphi\left(e_{31}\right)= a_5\langle 1, 0, \lambda_1\rangle+a_6\langle 0,1,\lambda_1+1\rangle,\\
&\varphi\left(e_{32}\right)= a_7\langle 0, 1, \lambda_1\rangle+a_8\langle 1,0,\lambda_1-1\rangle,\\
&\varphi\left(e_{12}\right)= a_3\langle 0, 0, \lambda_1-1\rangle, \;\varphi\left(e_{21}\right)= a_4\langle 0, 0, \lambda_1+1\rangle,\\
&\varphi\left(h_1\right)= a_1\langle 0, 0, \lambda_1\rangle, \;\varphi\left(h_2\right)= a_2\langle 0, 0, \lambda_1\rangle,\;
\varphi\left(e_{13}\right)= \varphi\left(e_{23}\right)= 0,
\end{align*}
where  $a_i\in \mathbb{F}$ for $1\leq i\leq 8$.

(4)  When $\lambda_1+\lambda_2=-2$,
\begin{align*}
&\varphi\left(e_{23}\right)= a_6\langle 1, 0, \lambda_1+1\rangle+a_7\langle 0, 1, \lambda_1+2\rangle,\\
&\varphi\left(e_{21}\right)= a_8\langle 1, 1, \lambda_1+2\rangle, \;\varphi\left(e_{31}\right)= \varphi\left(e_{32}\right)= 0,\\
&\varphi\left(h_1\right)= a_1\langle 1, 1, \lambda_1+1\rangle, \;\varphi\left(h_2\right)= a_2\langle 1, 1, \lambda_1+1\rangle,\\
&\varphi\left(e_{13}\right)= a_4\langle 1, 0, \lambda_1\rangle+a_5\langle 0, 1, \lambda_1+1\rangle, \; \varphi\left(e_{12}\right)= a_3\langle 1, 1, \lambda_1\rangle,
\end{align*}
where  $a_i\in \mathbb{F}$ for $1\leq i\leq 8$.

(5) When $\lambda_1+\lambda_2=-3$,
$$\varphi\left(e_{13}\right)= a\langle 1, 1, \lambda_1+1\rangle, \;
\varphi\left(e_{23}\right)= b\langle 1, 1, \lambda_1+2\rangle, \;\varphi(x)=0,$$
where $x=h_1, h_2, e_{ij}$ with $(i, j)\neq(1, 3), (2, 3)$, and $a, b\in \mathbb{F}$.

(6) When $\lambda_1+\lambda_2$ is not in $\{0, 1, -1, -2, -3\}$, $\varphi$ is 0.
\end{remark}

\section{First cohomology groups with coefficients in $\chi$-reduced Kac modules}

This section is devoted to the proof of Theorem \ref{1903221616}.

\subsection{ The case of $\chi$ being semi-simple}

In this section  we shall determine $\mathrm{H}^1(Z^{\chi}(\lambda))$ for $\chi$ being semi-simple.

\begin{proposition}\label{1807101113}
Let  $\chi$ be semi-simple, i.e.,
$$\chi\left(h_1\right)= r,\; \chi\left(h_2\right)= s,\;\chi\left(e_{12}\right)= 0,\; \chi\left(e_{21}\right)= 0.$$

(1) If $r\neq 0$ or $s\neq 0$, then $\mathrm{H}^1(Z^{\chi}(\lambda))=0.$

(2) If $r=s=0$ and $\lambda_1+\lambda_2$ is not in $\{-2, -3\}$, then $\mathrm{H}^1(Z^{\chi}(\lambda))=0.$

\end{proposition}
\begin{proof}
(1) If $\chi$ is semi-simple and  $r\neq 0$ or $s\neq 0$, then
$\lambda_1^p-\lambda_1=r^p\neq 0$ or $ \lambda_2^p-\lambda_2=s^p\neq 0$.
Hence  either $\lambda_1$ or $\lambda_2$ is not in $\mathbb{F}_p$.
Then by (\ref{1906211115}),   $Z^{\chi}(\lambda)$ and $\frak{g}$  have no common weight.
It follows that $\mathrm{H}^1(Z^{\chi}(\lambda))=0$ by  Lemma \ref{lem-weight-der+inner} and (\ref{1906131540}).

(2) If $r=s=0$, then $\lambda_1$ and $\lambda_2$ are in $\mathbb{F}_p$.
By Remark \ref{1904202222}, we only prove $\mathrm{H}^1(Z^{0}(\lambda))=0$ if $\lambda_1+\lambda_2=1,0,-1$.

When $\lambda_1+\lambda_2=-1$,
by Lemma \ref{lem-weight-der+inner}, we only consider the weight-derivation $\varphi$ satisfying Remark \ref{1904202222} (2).
Then $\varphi$ is odd.

\textit{Case  1: $\lambda_1= 0, 1, \ldots, (p-3)/2$.}
Then $\langle i, j, \lambda_1\rangle, \langle i, j, \lambda_1+1\rangle, \langle i, j, \lambda_1+2\rangle$
 are nonzero for $i, j=0$ or $1$. Especially, $\langle i, j, \lambda_1-1\rangle\neq 0$  for $\lambda_1=  1, \ldots, (p-3)/2$ but
 $\langle i, j, \lambda_1-1\rangle=0$ for $\lambda_1= 0$.
Since $\varphi$ is a derivation,
  we get
\begin{align*}
0=e_{21}\varphi(h_1)
&=a_{1}e_{21}\langle1,0,\lambda_1\rangle+a_2e_{21}\langle0,1,\lambda_1+1\rangle\\
&\stackrel{(\ref{e21s})}{=}(a_1-a_2)\langle1,0,\lambda_1+1\rangle+a_2\langle0,1,\lambda_1+2\rangle,\\
0=e_{21}\varphi(h_2)
&=a_{3}e_{21}\langle1,0,\lambda_1\rangle+a_4e_{21}\langle0,1,\lambda_1+1\rangle\\
&\stackrel{(\ref{e21s})}{=}(a_3-a_4)\langle1,0,\lambda_1+1\rangle+a_4\langle0,1,\lambda_1+2\rangle,\\
-a_{12}\langle1,1,\lambda_1\rangle&=\varphi\left([e_{12}, e_{31}]\right)\\
&=a_{11}e_{12}\langle1,1,\lambda_1+1\rangle+a_6e_{31}\langle0,1,\lambda_1\rangle\\
&\stackrel{(\ref{e31})(\ref{e12})}{=}\left(a_{11}(\lambda_1+1)^2+a_6\right)\langle1,1,\lambda_1\rangle,\\
a_7\langle0,0,\lambda_1\rangle&=\varphi\left([e_{12}, e_{23}]\right)\\
&=a_8e_{12}\langle0,0,\lambda_1+1\rangle+a_5e_{23}\langle1,0,\lambda_1-1\rangle+a_6e_{23}\langle0,1,\lambda_1\rangle\\
&\stackrel{(\ref{e12})(\ref{e23s})}{=}\left(a_{8}(\lambda_1+1)^2+a_5-a_6\right)\langle0,0,\lambda_1\rangle,\\
a_3\langle1,0,\lambda_1\rangle+a_4\langle0,1,\lambda_1+1\rangle&=\varphi\left([e_{23}, e_{32}]\right)\\
&=-a_{12}e_{23}\langle1,1,\lambda_1\rangle-a_8e_{32}\langle0,0,\lambda_1+1\rangle\\
&\stackrel{(\ref{e23s})(\ref{e32})}{=}(-a_{12}-a_8)\langle0,1,\lambda_1+1\rangle,\\
a_1\langle1,0,\lambda_1\rangle+a_2\langle0,1,\lambda_1+1\rangle&=\varphi\left([e_{13}, e_{31}]\right)\\
&=-a_{11}e_{13}\langle1,1,\lambda_1+1\rangle-a_7e_{31}\langle0,0,\lambda_1\rangle\\
&\stackrel{(\ref{e13})(\ref{e31})}{=}\left(a_{11}(\lambda_1+1)^2-a_7\right)\langle1,0,\lambda_1\rangle,\\
a_9\langle1,0,\lambda_1+1\rangle+a_{10}\langle0,1,\lambda_1+2\rangle&=\varphi\left([e_{23},e_{31}]\right)\\
&=-a_{11}e_{23}\langle1,1,\lambda_1+1\rangle-a_8e_{31}\langle0,0,\lambda_1+1\rangle\\
&\stackrel{(\ref{e23s})(\ref{e31})}{=}(a_{11}-a_8)\langle1,0,\lambda_1+1\rangle-a_{11}\langle0,1,\lambda_1+2\rangle.
\end{align*}
It follows that $a_i=0$ for $i=1, \ldots, 4$ and
\begin{align*}
 a_5&=-a_8\lambda_1(\lambda_1+2),  &  a_9&=-a_8-a_{10},  \\
a_6&=a_8+a_{10}(\lambda_1+1)^2,     &  a_{11}&=-a_{10},\\
a_7&=-a_{10}(\lambda_1+1)^2, &  a_{12}&=-a_8.
\end{align*}
Then we have $\varphi=a_8\frak{D}_{-\langle 1, 0, \lambda_1\rangle}+a_{10}\frak{D}_{\langle 0, 1, \lambda_1+1\rangle}$,
which is inner.
Then $\mathrm{H}^1(Z^{0}(\lambda))=0$ by Lemma \ref{lem-weight-der+inner} and (\ref{1906131540}).

 \textit{Case 2: $\lambda_1=p-1$.}
 Then $\langle i, j, \lambda_1\rangle, \langle i, j, \lambda_1+1\rangle$ and $\langle i, j, \lambda_1+2\rangle$
  are nonzero for $i, j=0$ or $1$.
Similarly to Case 1, by the definition of derivations and equations (\ref{h1})--(\ref{e23s}),
 those $a_i$  in Remark \ref{1904202222} (2) satisfy:
\begin{align*}
&a_7=0 & &\text{(by letting $x=e_{13}, y=e_{31}$),}\\
&a_8=0 & &\text{(by letting $x=e_{23}, y=e_{32}$),}\\
&a_1=a_2=0     &  &\text{(by letting $x=h_1, y=e_{31}$ or $e_{32}$),}\\
&a_3=a_4=0     &  &\text{(by letting $x=h_2, y=e_{31}$ or $e_{32}$),}\\
&a_9=a_{11}, a_{10}+a_{11}=0  &  &\text{(by letting $x=e_{21}, y=e_{31}$ or $e_{32}$),}\\
&a_6+a_{12}=0, a_{5}+a_{12}=0 &  &\text{(by letting $x=e_{12}, y=e_{31}$ or $e_{32}$).}
\end{align*}
It follows that $a_i=0$ for $i=1, \ldots, 4, 7, 8$ and
$$a_9=-a_{10}=a_{11},\; a_5=-a_{12}=a_{6}.$$
Then we have  $\varphi=a_6\frak{D}_{-\langle 1, 0, p-1\rangle}+a_{11}\frak{D}_{-\langle 0, 1, 0\rangle}$, which is inner.
Hence $\mathrm{H}^1(Z^{0}(\lambda))=0$ by Lemma \ref{lem-weight-der+inner} and (\ref{1906131540}).

\textit{Case 3: $\lambda_1=p-2$.}
 Then $$\langle i,j,\lambda_1\rangle=\langle i,j,\lambda_1+1\rangle=0, \;\langle i,j,\lambda_1+2\rangle\neq0, \;\langle i,j,\lambda_1-1\rangle\neq0,$$
where  $i,j=0$ or $1$.
Hence  it suffices to compute $a_5$ and $a_{10}$ in Remark \ref{1904202222} (2). Note that
$$a_5\langle 1,0,p-3\rangle=\varphi(e_{12})=\varphi([e_{13}, e_{32}])=0, $$
$$a_{10}\langle 0,1,0\rangle=\varphi(e_{21})=\varphi([e_{23}, e_{31}])=0,$$
which implies $a_5=a_{10}=0$, i.e., $\varphi=0$.
Hence $\mathrm{H}^1(Z^{0}(\lambda))=0$ by Lemma \ref{lem-weight-der+inner} and (\ref{1906131540}).

\textit{Case 4:  $\lambda_1=(p-1)/2, (p+1)/2, \ldots, p-3$.}
Then
$$\langle i, j, \lambda_1\rangle=\langle i, j, \lambda_1+1\rangle=\langle i, j, \lambda_1+2\rangle=\langle i, j, \lambda_1-1\rangle=0,$$
where $i, j=0$ or 1. That is, any weight-map satisfying Remark \ref{1904202222} (2) is 0.
Hence $\mathrm{H}^1(Z^{0}(\lambda))=0$ by Lemma \ref{lem-weight-der+inner} and (\ref{1906131540}).

Other cases of $\lambda_1+\lambda_2=0, 1$ can be treated similarly, which are omitted.
\end{proof}

Define the linear maps $\psi_k$ from $ \frak{g}$ to $Z^{\chi}(\lambda)$ as follows:
$$\psi_1\left(e_{13}\right)= \langle 1, 1, 0\rangle, \;\psi_1\left(e_{23}\right)= \langle 1, 1, 1\rangle,\; \psi_1\left(x\right)= 0,$$
where $x=h_1, h_2, e_{ij}$ with $(i, j)\neq(1, 3), (2, 3)$;
$$\psi_2\left(e_{23}\right)= \langle 0, 1, 0\rangle,\; \psi_2\left(e_{21}\right)= \langle 1, 1, 0\rangle,\; \psi_2\left(x\right)= 0,$$
where $x=h_1, h_2, e_{ij}$ with $(i, j)\neq(2,3), (2,1)$.
\begin{proposition}\label{1807101114}
(1) Let  $\chi=0$ and $\lambda_1+\lambda_2=-3$. Then
 $$\mathrm{H}^1(Z^{\chi}(\lambda))=\left\{\begin{array}{ll}
0,\;\;&\mbox{if}\;\;\lambda_1\neq p-1\\
\mathbb{F}\psi_1,\;\;&\mbox{if}\;\;\lambda_1=p-1.
\end{array}\right.$$
In particular,
$$\dim\mathrm{H}^1(Z^{\chi}(\lambda))=\left\{\begin{array}{ll}
0,\;\;&\mbox{if}\;\;\lambda_1\neq p-1\\
1,\;\;&\mbox{if}\;\;\lambda_1=p-1.
\end{array}\right.$$

(2) Let  $\chi=0$ and $\lambda_1+\lambda_2=-2$. Then
 $$\mathrm{H}^1(Z^{\chi}(\lambda))=\left\{\begin{array}{ll}
0,\;\;&\mbox{if}\;\;\lambda_1\neq p-2\\
\mathbb{F}\psi_2,\;\;&\mbox{if}\;\;\lambda_1=p-2.
\end{array}\right.$$
In particular,
$$\dim\mathrm{H}^1(Z^{\chi}(\lambda))=\left\{\begin{array}{ll}
0,\;\;&\mbox{if}\;\;\lambda_1\neq p-2\\
1,\;\;&\mbox{if}\;\;\lambda_1=p-2.
\end{array}\right.$$
\end{proposition}
\begin{proof}
 When $\chi=0$ and $\lambda_1+\lambda_2=-3$,
 by Lemma \ref{lem-weight-der+inner}, we only consider the weight-derivation $\varphi$ satisfying Remark \ref{1904202222} (5).
Then $\varphi$ is odd.

\textit{Case 1:  $\lambda_1=0, 1, \ldots, (p-5)/2$ or $\lambda_1=p-2$.}
Then $\langle i, j, \lambda_1\rangle, \langle i, j, \lambda_1+1\rangle$ and $\langle i, j, \lambda_1+2\rangle$ are nonzero for $i,j=0$ or $1$.
Hence we have
 \begin{align*}
 0=\varphi([e_{13}, e_{23}])&=-be_{13}\langle1,1,\lambda_1+2\rangle-ae_{23}\langle1,1,\lambda_1+1\rangle\\
 &\stackrel{(\ref{e13})(\ref{e23s})}{=}(a-b)\langle 0, 1, \lambda_1+2\rangle+\left(a-(\lambda_1+2)^2b\right)\langle 1, 0, \lambda_1+1\rangle,\\
0=\varphi([e_{12}, e_{13}])&=ae_{12}\langle1,1,\lambda_1+1\rangle\\
&\stackrel{(\ref{e12})}{=}a(\lambda_1+1)(\lambda_1+3)\langle 1, 1, \lambda_1\rangle.
\end{align*}
 It implies that $a=b=b(\lambda_1+2)^2$ and $a(\lambda_1+1)(\lambda_1+3)=0.$
Then $a=b=0$, since $\lambda_1$ is not  in $\{p-1, p-3\}$,
 i.e., $\varphi=0$. By Lemma \ref{lem-weight-der+inner} and (\ref{1906131540}), $\mathrm{H}^1\left(Z^{0}(M)\right)=0$.

\textit{Case 2: $\lambda_1=(p-3)/2, (p-1)/2, \ldots, p-3$.}
Then $\langle 1, 1, \lambda_1+1\rangle=\langle 1, 1, \lambda_1+2\rangle=0$.
 We have  $\varphi=0$ satisfying  Remark \ref{1904202222} (5).
 Hence by Lemma \ref{lem-weight-der+inner} and (\ref{1906131540}), $\mathrm{H}^1\left(Z^{0}(M)\right)=0$.

\textit{Case 3: $\lambda_1=p-1$.}
Then
$\langle i, j, \lambda_1+1\rangle, \langle i, j, \lambda_1+2\rangle$  are nonzero  for $i,j=0$ or 1.
On one hand, by the definition of derivations we have
 $$0=\varphi([e_{13}, e_{23}])\stackrel{(\ref{e13})(\ref{e23s})}{=}(a-b)\langle 0, 1, 1\rangle+(a-b)\langle 1, 0, 0\rangle,$$
 i.e., $a=b.$
 On the other hand,  $\psi_1$ is  an odd weight-derivation if $\chi=0$ and $\lambda=(p-1, p-2)$.
Then  $\varphi=a\psi_1$ for $a\in \mathbb{F}$.
By   (\ref{e13}), we get
\begin{eqnarray*}
&&e_{13}.\langle 0,0,k\rangle=0, \\
&&e_{13}.\langle 1,0,k\rangle=(\lambda_1-k)\langle 0,0,k\rangle,\\
&&e_{13}.\langle 0,1,k\rangle=k(\lambda_1-\lambda_2-k+1)\langle 0,0,k-1\rangle,\\
&&e_{13}.\langle 1,1,k\rangle=(\lambda_1-k+1)\langle0,1,k\rangle-k(\lambda_1-\lambda_2-k+1)\langle 1,0,k-1\rangle.
\end{eqnarray*}
So $\psi_1$ is not inner. Then by Lemma \ref{lem-weight-der+inner} and (\ref{1906131540}),  we have $\dim\mathrm{H}^1(Z^{0}(\lambda))=1$.

The case when $\chi=0$ and $\lambda_1+\lambda_2=-2$ can be treated similarly, which is omitted.
\end{proof}

By Propositions \ref{1807101113} and  \ref{1807101114}, we get the following theorem.
\begin{theorem}\label{1903231619}
Let  $\chi$ be semi-simple. Then
$$\mathrm{H}^1(Z^{\chi}(\lambda))=\left\{\begin{array}{ll}
\mathbb{F}\psi_1,\;\;&\mbox{if}\;\;\chi=0,\;\;\mbox{and}\;\;\lambda=(p-1, p-2)\\
\mathbb{F}\psi_2,\;\;&\mbox{if}\;\;\chi=0,\;\;\mbox{and}\;\;\lambda=(p-2, 0)\\
0,\;\;&\mbox{otherwise}.
\end{array}\right.$$
In particular,
$$\dim\mathrm{H}^1(Z^{\chi}(\lambda))=\left\{\begin{array}{ll}
1,\;\;&\mbox{if}\;\;\chi=0,\;\;\mbox{and}\;\;\lambda=(p-1, p-2)\;\;\mbox{or}\;\;(p-2, 0)\\
0,\;\;&\mbox{otherwise}.
\end{array}\right.$$
\end{theorem}

\subsection{ The case of $\chi$ being nilpotent}
In this section  we shall determine $\mathrm{H}^1(Z^{\chi}(\lambda))$ for $\chi$ being nilpotent.
\begin{theorem}\label{1807101119}
Let  $\chi$ be nilpotent, i.e., $$\chi\left(h_1\right)=\chi\left(h_2\right)= r, \chi\left(e_{12}\right)= 0, \chi\left(e_{21}\right)= 1.$$
Then  $\mathrm{H}^1(Z^{\chi}(\lambda))=0.$
\end{theorem}
\begin{proof}
 If   $r\neq 0$, then $\lambda_1^p-\lambda_1=r^p\neq 0, \lambda_2^p-\lambda_2=r^p\neq 0.$
Hence  both $\lambda_1$ and $\lambda_2$ are not in $\mathbb{F}_p$.
By (\ref{1906211115}),  $Z^{\chi}(\lambda)$ and $\frak{g}$ have no common weight.
It follows that $\mathrm{H}^1(Z^{\chi}(\lambda))=0$ by  Lemma \ref{lem-weight-der+inner} and (\ref{1906131540}).

If $r=0$ and $\lambda_1+\lambda_2$ is not in $ \{0, 1, -1, -2, -3\}$,
then any weight-map is zero by Remark \ref{1904202222}. It follows that $\mathrm{H}^1(Z^{\chi}(\lambda))=0$ by Lemma \ref{lem-weight-der+inner} and (\ref{1906131540}).

If $r=0$ and $\lambda_1+\lambda_2=1$, then by  Lemma \ref{lem-weight-der+inner}, we only consider the weight-derivation $\varphi$
satisfying  Remark \ref{1904202222} (1), which is odd.
Note that
$$0=e_{31}\varphi(e_{31})\stackrel{(\ref{e31})}{=}a\langle1,0,\lambda_1\rangle, \; 0=e_{32}\varphi(e_{32})\stackrel{(\ref{e32})}{=}b\langle0,1,\lambda_1-1\rangle,$$
which implies  $a=b=0$.
Then any weight-derivation satisfying  Remark \ref{1904202222} (1) must be  zero.

The other cases when $r=0$ and $\lambda_1+\lambda_2\in \{0, -1, -2, -3\}$ can be similarly treated, which are omitted.
\end{proof}

Theorem \ref{1903221616} follows from   Theorems \ref{1903231619} and \ref{1807101119}.

\section{First cohomology groups with coefficients in simple $\chi$-modules}

This section is devoted to the proof of Theorem \ref{1903231610}.

\begin{proposition}\label{1903231609}
(1) If $\lambda_1\neq p-1$ and $\lambda_2\neq 0$, then $\mathrm{H}^1(S^{\chi}(\lambda))=0.$

(2) Let $\chi=0$ or $\chi$ be nilpotent with $\chi(h_1)=0$. If $\lambda$ is not in the set
$$\left\{(p-1, 1), (p-1, 2), (p-1, 0), (p-1, p-2), (p-1, p-1), (0, 0), (1, 0), (p-2, 0), (p-3, 0)\right\},$$
then $\mathrm{H}^1(S^{\chi}(\lambda))=0.$

(3) If $\chi$ is semi-simple but nonzero, or $\chi$ is nilpotent with $\chi(h_1)\neq 0$, then $\mathrm{H}^1(S^{\chi}(\lambda))=0.$
\end{proposition}
\begin{proof}
(1) By Lemma \ref{2002081211} (1), we have $Z^{\chi}(\lambda)=S^{\chi}(\lambda)$  if $\lambda_1\neq p-1$ and $\lambda_2\neq 0$.
Then by Theorem \ref{1903221616}, we get $\mathrm{H}^1(S^{\chi}(\lambda))=0.$

(2) Suppose that $\chi=0$ or $\chi$ is nilpotent with $\chi(h_1)=0$.
Each basis vector of $S^{\chi}(\lambda)$ in Lemma \ref{2002081211} (2)--(6)
is also a weight vector.
Then by (\ref{1906211115}), if $\lambda$ is a common weight of $S^{\chi}(\lambda)$ and $\frak{g}$,
then $\lambda$ must be in the set
$$\left\{(p-1, 1), (p-1, 2), (p-1, 0), (p-1, p-2), (p-1, p-1), (0, 0), (1, 0), (p-2, 0), (p-3, 0)\right\}.$$
Hence by Lemma \ref{lem-weight-der+inner} and (\ref{1906131540}), we get (2).

(3) If $\chi$ is semi-simple but nonzero or $\chi$ is nilpotent with $\chi(h_1)\neq 0$,
then either $\lambda_1$ or $\lambda_2$ is not in $\mathbb{F}_p$.
Then by (\ref{1906211115}),   $S^{\chi}(\lambda)$ and $\frak{g}$ have no common weight.
It follows that $\mathrm{H}^1(S^{\chi}(\lambda))=0$ by  Lemma \ref{lem-weight-der+inner} and (\ref{1906131540}).
\end{proof}

If $\chi$ is semi-simple, then we only need to consider the cases of $\chi=0$ and $\lambda$ being in
$$\left\{(p-1, 1), (p-1, 2), (p-1, 0), (p-1, p-2), (p-1, p-1), (0, 0), (1, 0), (p-2, 0), (p-3, 0)\right\}.$$

Define the linear maps $\psi_k$ from $ \frak{g}$ to $S^{\chi}(\lambda)$ by letting
$$\psi_3\left(e_{13}\right)=-\langle 0, 1, 0\rangle, \;\psi_3\left(e_{23}\right)=\langle 1, 0, 0\rangle, \psi_3\left(x\right)=0,$$
where $x=h_1, h_2, e_{ij}$ with $(i, j)\neq(1, 3), (2, 3);$
\begin{align*}
&\psi_4\left(h _1\right)= -\langle 1, 0, p-1\rangle,\; \psi_4\left(e_{12}\right)= \langle 1, 0, p-2\rangle,\\
&\psi_4\left(e_{13}\right)= \langle 0, 0, p-1\rangle,\; \;\psi_4\left(x\right)= 0,
\end{align*}
where $x=h_2, e_{ij}$ with $(i, j)\neq(1, 2), (1, 3);$
$$\psi_5\left(e_{23}\right) =\langle 0, 0, 0\rangle,\; \psi_5\left(e_{21}\right)= -\langle 1, 0, 0\rangle,\; \psi_5\left(x\right)= 0,$$
where $x=h_1, h_2, e_{ij}$ with $ (i, j)\neq(2, 1), (2, 3)$;
$$\psi_6\left(e_{12}\right)= \langle 0, 0, p-2\rangle,\; \psi_6\left(e_{32}\right)=\langle 1, 0, p-2\rangle, \;\psi_6\left(x\right)= 0,$$
where $x=h_1, h_2, e_{ij}$ with $ (i, j)\neq (1, 2), (3, 2)$;
$$\psi_7\left(e_{21}\right)= \langle 0, 0, 0\rangle,\; \psi_7\left(e_{31}\right)= \langle 0, 1, 0\rangle,\; \psi_7\left(x\right)= 0,$$
where $x=h_1, h_2, e_{ij}$ with $(i, j)\neq (2, 1), (3, 1)$;
$$\psi_8\left(e_{32}\right)= \langle 0, 0, 0\rangle, \;\psi_8\left(e_{31}\right)= -\langle0,0,1\rangle,\; \psi(x)=0$$
where $x=h_1, h_2, e_{ij}$ with $(i,j)\neq (3,1), (3,2).$

\begin{proposition}\label{1903231605}
Let $\chi=0.$ Then

(1)
$$\mathrm{H}^1(S^{\chi}(\lambda))=\left\{\begin{array}{ll}
\mathbb{F}\psi_6+\mathbb{F}\psi_7,\;\;&\mbox{if} \; \lambda=(p-1, 1)\\
\mathbb{F}\psi_4+\mathbb{F}\psi_5,\;\;&\mbox{if} \; \lambda=(p-1, 0).
\end{array}\right.$$
In particular, if $\lambda=(p-1, 1)$ or $(p-1, 0)$, then $\dim\mathrm{H}^1(S^{\chi}(\lambda))=2.$

(2)
$$\mathrm{H}^1(S^{\chi}(\lambda))=\left\{\begin{array}{ll}
\mathbb{F}\psi_{8},\;\;&\mbox{if} \; \lambda=(1, 0)\\
\mathbb{F}\psi_3,\;\;&\mbox{if} \; \lambda=(p-1, p-1).
\end{array}\right.$$
In particular, if $\lambda=(1, 0)$ or $(p-1, p-1)$, then $\dim\mathrm{H}^1(S^{\chi}(\lambda))=1.$

(3) If $\lambda=(p-1, 2), (p-1, p-2), (p-3, 0), (0, 0)$ or $(p-2, 0)$, then $\dim\mathrm{H}^1(S^{\chi}(\lambda))=0.$
\end{proposition}
\begin{proof}
(1) \textit{Case 1: $\lambda=(p-1, 1)$.}
By Lemma \ref{2002081211} (2), $S^{0}(\lambda)$ has a basis $\langle 0, 0, k\rangle, \langle 1, 0, k\rangle, \langle 0, 1, 0\rangle$ with $0\leq k\leq p-2$  and
 $\langle i,j,p-1\rangle=0$ with $i,j=0$ or 1.
Then by  Remark \ref{1904202222} (3), we only consider the weight-derivation
$\varphi: \frak{g}\longrightarrow S^{0}(\lambda)$ satisfying
\begin{equation*}
\begin{split}
&\varphi\left(e_{12}\right)= a_3\langle 0, 0, p-2\rangle, \;\varphi\left(e_{32}\right)= a_8\langle 1, 0, p-2\rangle, \\
&\varphi\left(e_{21}\right)= a_4\langle 0, 0, 0\rangle, \;\varphi\left(e_{31}\right)= a_6\langle 0, 1, 0\rangle, \;\varphi\left(x\right)= 0,
\end{split}
\end{equation*}
where $x=h_1, h_2, e_{13}, e_{23}$. Hence $\varphi$ is even.
On one hand, since $\varphi$ is a derivation, we get
$$a_4\langle0,0,0\rangle=\varphi([e_{23}, e_{31}])=a_6e_{23}\langle0,1,0\rangle\stackrel{(\ref{e23s})}{=}a_6\langle0,0,0\rangle,$$
$$-a_8\langle1,0,p-2\rangle=\varphi([e_{12}, e_{31}])=a_6e_{12}\langle0,1,0\rangle-a_3e_{31}\langle0,0,p-2\rangle\stackrel{(\ref{e31})(\ref{e12})}{=}-a_3\langle1,0,p-2\rangle.$$
By  equations (2.5), (2.7) and (2.10), we get
$$-a_8\langle1,0,p-2\rangle=-a_3\langle1,0,p-2\rangle,\; a_4\langle0,0,0\rangle=a_6\langle0,0,0\rangle,$$
i.e.,
$a_8=a_3$ and $a_6=a_4$.
On the other hand,   $\psi_6$ and $\psi_7$ defined above are weight-derivations if $\chi=0$ and $\lambda=(p-1, 1)$.
Then $\varphi=a_3\psi_6+a_4\psi_7$.
Assert that any nontrivial linear combination of $\psi_6$ and  $\psi_7$ is not an inner derivation.
In fact, if the even derivation $a\psi_6+b\psi_7$ is inner, then  there exists $m\in S^{0}(\lambda)$ such that
$(a\psi_6+b\psi_7)(x)=(-1)^{|x||m|}xm$ for all $x\in \frak{g}$. In particular, $(a\psi_6+b\psi_7)(h_i)=h_im$.
By the definitions of $\psi_6$ and $\psi_7$, we have $(a\psi_6+b\psi_7)(h_i)=0$ and hence $h_im=0$ for $i=1, 2$.
Note that $S^{0}(\lambda)$ has no weight vector of  weight $(0, 0)$
in the case when  $\lambda=(p-1, 1)$ by Lemma \ref{2002081211} (2) and equations (\ref{h1}), (\ref{h2}).
It follows that $m=0$, i.e., $a\psi_6+b\psi_7=0$.
Considering that $\psi_6$ and  $\psi_7$ are linearly independent, we have $a=b=0$.
Then by Lemma \ref{lem-weight-der+inner} and (\ref{1906131540}),
we get
$$\mathrm{H}^1(S^{0}(\lambda))=\mathbb{F}\psi_6+\mathbb{F}\psi_7,\; \dim\mathrm{H}^1(S^{0}(\lambda))=2.$$

\textit{Case 2: $\lambda=(p-1, 0)$.}
By Lemma \ref{2002081211} (3) and Remark \ref{1904202222} (2),  we only consider the weight-derivation
$\varphi: \frak{g}\longrightarrow S^{0}(\lambda)$ satisfying
\begin{alignat*}{3}
&\varphi\left(e_{23}\right)= a_8\langle 0, 0, 0\rangle,   & &\varphi\left(e_{21}\right)= (a_9+a_{10})\langle 1, 0, 0\rangle,  \\
&\varphi\left(e_{13}\right)= a_7\langle 0, 0, p-1\rangle, \; & &\varphi\left(e_{12}\right)= (a_5-a_6)\langle 1, 0, p-2\rangle, \\
&\varphi\left(h_{1}\right)= a_1\langle 1, 0, p-1\rangle, & &\varphi\left(h_{2}\right)= a_3\langle 1, 0, p-1\rangle,\\
&\varphi\left(e_{31}\right)=\varphi\left(e_{32}\right)= 0.
\end{alignat*}
Therefore $\varphi$ is odd.
Since $\varphi$ is a derivation,
 we have
 $$a_3\langle1,0,p-1\rangle=\varphi\left([e_{23}, e_{32}]\right)=-a_8e_{32}\langle0,0,0\rangle\stackrel{(\ref{e32})}{=}-a_8\langle0,1,0\rangle,$$
 $$(a_9+a_{10})\langle1,0,0\rangle=\varphi\left([e_{23}, e_{31}]\right)=-a_8e_{31}\langle0,0,0\rangle\stackrel{(\ref{e31})}{=}-a_8\langle1,0,0\rangle,$$
 $$a_1\langle1,0,p-1\rangle=\varphi\left([e_{13}, e_{31}]\right)=-a_7e_{31}\langle0,0,p-1\rangle\stackrel{(\ref{e31})}{=}-a_7\langle1,0,p-1\rangle,$$
 $$0=\varphi\left([e_{12}, e_{13}]\right)=a_7e_{12}\langle 0, 0, p-1\rangle+(a_5-a_6)e_{13}\langle 1, 0, p-2\rangle\stackrel{(\ref{e12})(\ref{e13})}{=}(a_5-a_6-a_7)\langle0,0,p-2\rangle.$$
On one hand,  it is true $\langle0,1,0\rangle=0$ by Lemma \ref{2002081211} (3).
Then we get $a_3=0$ from the first equation.
On the other hand, from the last three equations we have
$$a_5-a_6=a_7=-a_1,\; a_9+a_{10}=-a_8.$$
Note that $\psi_4$ and $\psi_5$ are weight-derivations if $\chi=0$ and $\lambda=(p-1, 0)$.
Then $\varphi=-a_1\psi_4+a_8\psi_5$.
Assert that any nontrivial linear combination of $\psi_4$ and  $\psi_5$ is not an inner derivation.
In fact, if the odd derivation $a\psi_4+b\psi_5$ is inner, then  there exists $m\in S^{0}(\lambda)$ such that
$(a\psi_4+b\psi_5)(x)=(-1)^{|x||m|}xm$ for all $x\in \frak{g}$.
By the definitions of $\psi_4$ and $\psi_5$, we get
$$(a\psi_4+b\psi_5)(h_1)=-a\langle1, 0, p-1\rangle, \;(a\psi_4+b\psi_5)(h_2)=0. $$
Then by $(a\psi_4+b\psi_5)(h_i)=h_im$ for $i=1, 2$, we have $h_1m=-a\langle1, 0, p-1\rangle, h_2m=0$.
Note that $\langle1, 0, p-1\rangle$ is of weight $(0, 0)$ if $\lambda=(p-1, 0)$.
Hence $a=0$ and $m$ is in the weight space $S^{\chi}(\lambda)_{(0,0)}$.
By  (\ref{h1}), (\ref{h2}) and Lemma \ref{2002081211}, we get  $m\in \mathbb{F}\langle1,0,p-1\rangle$.
Then by
$$b\langle0, 0, 0\rangle=b\psi_5(e_{23})=-e_{23}m\stackrel{(\ref{e23s})}{=}0,$$
 we have     $b=0$.
Consequently,  by Lemma \ref{lem-weight-der+inner} and (\ref{1906131540}),
we get  $\mathrm{H}^1(S^{0}(\lambda))=\mathbb{F}\psi_4+\mathbb{F}\psi_5.$
Considering that $\psi_4$ and  $\psi_5$ are linearly independent, we get
 $\dim\mathrm{H}^1(S^{0}(\lambda))=2$.

(2) We only prove the case of $\lambda=(1, 0)$ in the following, the other case of $\lambda=(p-1, p-1)$ can be treated similarly.
Suppose $\lambda=(1, 0)$. Note that
 $S^{0}(\lambda)$ has a basis
 $\langle 0, 0, 0\rangle, \langle 0, 0, 1\rangle, \langle 1, 0, 0\rangle$
 by Lemma \ref{2002081211} (5).
Then by  Remark \ref{1904202222} (1),  we only consider the weight-derivation
$\varphi: \frak{g}\longrightarrow S^{0}(\lambda)$ satisfying
$$\varphi\left(e_{32}\right)= b\langle 0, 0, 0\rangle, \;\varphi\left(e_{31}\right)= a\langle 0, 0, 1\rangle,\; \varphi\left(x\right)= 0,$$
where $x=h_1, h_2, e_{ij}$ with $(i, j)\neq(3, 1), (3, 2).$
On one hand, by the definition of  derivations and $\langle0,1,1\rangle=\langle1, 0,0\rangle$ in $S^{0}(\lambda)$, we have
$$0=\varphi([e_{31}, e_{32}])\stackrel{(\ref{e31})(\ref{e32})}{=}(-b-a)\langle1,0,0\rangle,$$
which implies $ a+b=0.$
On the other hand,  $\psi_8$ is a weight-derivation if $\chi=0$ and $\lambda=(1, 0)$.
Then we get $\varphi=b\psi_{8}$.
 Note that $\psi_{8}$ is not inner.
Hence by Lemma \ref{lem-weight-der+inner} and (\ref{1906131540}),  we get
$$\mathrm{H}^1(S^{0}(\lambda))=\mathbb{F}\psi_{8},\;\;\dim\mathrm{H}^1(S^{0}(\lambda))=1.$$

(3) We only consider the case of $\lambda=(p-1, p-2)$, since the other cases can be treated similarly.
Suppose $\lambda=(p-1, p-2)$.
Note that  $\langle 1, 1, k\rangle=0$  for $k=0,1,\ldots,p-1$ in $S^{0}(\lambda)$ by Lemma \ref{2002081211} (2).
Then any weight-derivation  satisfying  Remark \ref{1904202222} (5) is zero.
It follows that  $\mathrm{H}^1(S^{0}(\lambda))=0$.

\end{proof}

As in the case of $\chi$ being nilpotent, we only need to consider ones of $\chi(h_1)=0$ and $\lambda$ being in
$$\left\{(p-1, 1), (p-1, 2), (p-1, 0), (p-1, -2), (p-1, p-1), (0, 0), (1, 0), (p-2, 0), (p-3, 0)\right\}.$$

\begin{proposition}\label{1903231616}
Let $\chi$ be nilpotent with $\chi(h_1)=0$. If $\lambda$ is in
$$\left\{(p-1, 1), (p-1, 2), (p-1, 0), (p-1, p-2), (p-1, p-1), (0, 0), (1, 0), (p-2, 0), (p-3, 0)\right\},$$
then $\mathrm{H}^1(S^{\chi}(\lambda))=0.$
\end{proposition}
\begin{proof}
We only prove the case of $\lambda=(0, 0)$, since the other cases can be treated similarly.
Suppose $\lambda=(0,0)$. By Lemma \ref{2002081211} (6),
$S^{\chi}(\lambda)$ has a basis $\langle 0, 0, k\rangle, \langle 1, 0, k\rangle$
and
$$\langle 0, 1, k\rangle=k\langle 1, 0, k-1\rangle, \;\langle 1, 1, k\rangle=0,$$
where $ 0\leq k\leq p-1$.
Then by  Remark \ref{1904202222} (3)  we only consider the weight-derivation  $\varphi$ satisfying
 \begin{alignat*}{3}
&\varphi\left(h_{1}\right)= a_1\langle 0, 0, 0\rangle, &  & \varphi\left(h_{2}\right)= a_2\langle 0, 0, 0\rangle,  \\
&\varphi\left(e_{12}\right)= a_3\langle 0, 0, p-1\rangle, &  &\varphi\left(e_{13}\right)=\varphi\left(e_{23}\right)= 0,\\
&\varphi\left(e_{32}\right)= a_8\langle 1,0,p-1\rangle, \;& & \varphi\left(e_{21}\right)= a_4\langle 0, 0, 1\rangle,\\
&\varphi\left(e_{31}\right)=(a_5+a_6)\langle 1, 0, 0\rangle.
\end{alignat*}
By the definition of derivations and $\langle0,1,p-1\rangle=-\langle1,0,p-2\rangle$, we get
\begin{align*}
&0=e_{21}\varphi(h_1)=a_1e_{21}\langle0,0,0\rangle\stackrel{(\ref{e21n})}{=}a_1\langle0,0,1\rangle,\\
&0=e_{21}\varphi(h_2)=a_2e_{21}\langle0,0,0\rangle\stackrel{(\ref{e21n})}{=}a_2\langle0,0,1\rangle,\\
&\varphi\left([e_{23}, e_{32}]\right)=a_8e_{23}\langle1,0,p-1\rangle\stackrel{(\ref{e23n})}{=}a_8\langle0,0,0\rangle,\\
&\varphi\left([e_{12}, e_{21}]\right)=a_4e_{12}\langle0,0,1\rangle-a_3e_{21}\langle0,0,p-1\rangle\stackrel{(\ref{e12})(\ref{e21n})}{=}-a_3\langle0,0,0\rangle,\\
&\varphi\left([e_{21}, e_{31}]\right)=(a_5+a_6)e_{21}\langle1,0,0\rangle-a_4e_{31}\langle0,0,1\rangle\stackrel{(\ref{e31})(\ref{e21n})}{=}(a_5+a_6-a_4)\langle1,0,1\rangle.
\end{align*}
It follows that $a_i=0$ for $i=1,2,3,8$, and $a_5+a_6=a_4$.
Then by $\langle0,1,0\rangle=0$, we get $\varphi=a_4\frak{D}_{\langle0, 0, 0\rangle}$,
which is inner.
By Lemma \ref{lem-weight-der+inner} and (\ref{1906131540}) we have $\mathrm{H}^1(S^{\chi}(\lambda))=0$.

\end{proof}

Theorem \ref{1903231610} follows from Propositions \ref{1903231609}--\ref{1903231616}.

\noindent \textbf{Acknowledgements}

The authors thank the referee  for the careful reading and valuable suggestions.

\end{document}